\documentclass[oneside,reqno,12pt]{amsart}
\usepackage[left=1in,top=1in,right=1in,bottom=1in]{geometry}
\usepackage{amsmath,amsfonts,amssymb,amsthm}
\usepackage{verbatim}
\usepackage{latexsym}
\usepackage[colorlinks,citecolor=red,pagebackref,hypertexnames=false]{hyperref}

\newtheorem{theorem}{Theorem}[section]
\newtheorem{lemma}[theorem]{Lemma}

\newtheorem{definition}[theorem]{Definition}
\newtheorem{corollary}[theorem]{Corollary}

\newtheorem{example}[theorem]{Example}

\newtheorem{claim}[theorem]{Claim}

\newcommand\RR{{\Bbb R}}
\newcommand\CC{{\Bbb C}}

\newcommand\NN{{\Bbb N}}
\newcommand\ZZ{{\Bbb Z}}

 \newcommand\HH{{\Bbb H}}

\newcommand{\note}[2][\null]{%
  \marginpar{\renewcommand{\baselinestretch}{1}\vspace{-1em}\hrule\vspace{3pt}%
  \tiny\raggedright{#2\ifx#1\null\else\\\hfill---
  {\em #1}\fi}\vspace{1.5em}}%
}

\begin{document}
\title{A unified approach for Littlewood-Paley Decomposition of Abstract Besov Spaces}

\author{Azita Mayeli}
\address{Department of Mathematics and Computer Sciences,   City University of New York, Queensborough, Bayside,  222-05 56th Ave.,  
 New York, NY 11364}
\email{amayeli@qcc.cuny.edu}
\subjclass[2010]{43, 46 (primary), and 41  (secondary)} 
\date{\today}

   \begin{abstract} 
We apply 
    spectral theoretic methods  to obtain a Littlewood-Paley decomposition of 
    abstract inhomogeneous  Besov  spaces  in terms of \lq\lq{}smooth\rq\rq{} and \lq\lq{}bandlimited\rq\rq{} functions.  Well-known decompositions in several contexts are as special examples and are unified under the spectral theoretic  approach.
      \end{abstract}

  \maketitle
 \section{Introduction}\label{introduction}

Besov spaces appear in many subfields of analysis and applied mathematics. 
In the classical setting, the  Besov space $B_q^\alpha(L^p)$ is the set of functions in $L^p$ with smoothness degree $\alpha$ and   (quasi)norm  is controlled by $q$. 
There are two types of definitions for these spaces. One type uses Fourier transform (for example see \cite{Peetre}), and the other uses modulus of continuity.  The Besov spaces defined by modulus of continuity are more practical in many areas of analysis, for example, in approximation theory. (See \cite{D-P}.) \\

To understand the structure of the Besov spaces for application purposes, it is natural to decompose a Besov function  into simple building blocks and hereby to reduce the study of functions to study of only  the elements in the decomposition.   Wavelet and frame theory have been very useful tools to achieve this goal.    A unified characterizations of Besov spaces in terms of atomic decomposition using group representation theoretic approach was given
by Feichtinger and Gr\"ochenig (\cite{FG}). In the classical level, this kind of decomposition using spectral theoretic approach was proved in  \cite{FJ}. 
 New results in this direction in the context of Lie groups and homogeneous manifolds  were recently published in \cite{CO1}-\cite{CMO1}, \cite{Fueh,FM1}, and \cite{gm1}-\cite{gp}.  In   \cite{gm2, gm3}, the authors 
 constructed  continuous and time-frequency localized wavelets and applied them to the classification of Besov spaces on the compact Riemannian manifolds.  \\

The purpose of the present paper is to describe the norm of abstract Besov spaces $B_q^\alpha(\mathcal H)$ when $L^p$ is replaced by abstract Hilbert space $\mathcal H$.  Therefore other well-known descriptions of $B_q^\alpha(L^2)$  are considered  as examples of our theory.
The abstract Besov spaces were introduced using modulus of continuity, for example, by Lion in \cite{Lion}.  
We   use  the  definition in \cite{Lion} and establish our results by developing  a connection between ``frequency content\rq\rq{} of  vectors in $\mathcal H$  and their  ``smoothness\rq\rq{}. We show that a vector (function) belongs to $B_q^\alpha(\mathcal H)$ if and only if the sequence of its  ``filtered\rq\rq{}   versions   satisfies some certain rate of norm convergence  in $l^q$. By this, we can identify  the Besov space $B_q^\alpha(\mathcal H)$ with the sequence space $l^q$, similar to the identification of a function with its  Fourier coefficients. 
While the idea behind such identification is simple, but the proofs are technical and main difficulties arose when $L^2$ is replaced by an abstract Hilbert space. 
 The main result of this paper appears in  Theorem \ref{mainLemma1} and its proof is given in Section \ref{Proof of main Theorem}.  
 
\section*{Acknowledgements}  The author would like to  thank Professor  Ronald R.  Coifman for many helpful discussions and suggestions to improve the results of this paper. 

\section{Preliminaries}\label{notations}
  Let   $\mathcal H$ be a Hilbert space and $A$ be 
  any  selfadjoint positive definite operator  in $\mathcal{H}$ whose domain is dense  in $\mathcal H$.  
  The domain of $A$  contains all functions $f\in \mathcal H$ for which $\|Af\|$ is finite.  Let $u: \RR^+\to \RR$ be a   positive function  such that  $u\in C^\infty(\RR^+)$,  $u^{(n)}$ decays rapidly at infinity for all $n\in \NN$,   and $\lim_{\xi\to 0^+} u^{(n)}(\xi)$ exists.   Therefore  $u$ is a Schwartz function on $\RR^+$.  Moreover, we assume that $u(s+t)= u(s)u(t)$ and $u(0)=1$. 
  For  $t>0$, we define the operator  $T_t=u_t(A)= u(tA)$   mapping  $\mathcal H$  to $\mathcal H$.    Then $\{T_t\}_{t>0}$ satisfies the semigroup axioms $T_{t+s}= T_tT_s$ and $T_0=\mathrm{identity}$.  Let   $M:=\|u\|_{\infty}$. Then $M<\infty$ and  we have the following. 
  \begin{enumerate}
  \item $\lim_{t\to 0^+} T_tf = f$ for all $f \in\mathcal H$
  \item  $T_t$ is symmetry (self-adjoint), 
  and 
\item    $\|T_tf\|\leq M \|f\|$ for all $f\in \mathcal H$.
\end{enumerate}
Under the construction of the semigroup $\{T_t\}$, the operator  
  $A$ is the infinitesimal generator of semigroup $\{T_t\}_{t>0}$. That means, in the Hilbert space norm 
  $$\lim_{t\to 0^+} \left\|\cfrac{T_tf -f}{t} - Af\right\|= 0$$
   
   We define {\it the modulus of continuity} for 
 $0<q\leq \infty$, $\alpha>0$, and $r\in \NN$ with $r\geq \alpha$,   by
\begin{equation}\notag
\Omega_{r}(s,\>f )=\sup_{0<\tau\leq s} \| \left( I-T_\tau \right)^r f\|~.
\end{equation}
   Following \cite{Lion},   
the   {\it inhomogenous abstract Besov space} $B_q^\alpha: = B_q^\alpha(\mathcal H, A)$ contains  all functions $f\in \mathcal H$ for which   
  \begin{align}\notag
   \int_0^1 \left(s^{-\alpha}  \Omega_{r}(s,\>f )\right)^q   \ \cfrac{ds}{s} <\infty . 
  \end{align}

  The space $B_q^\alpha$, $1\leq q<\infty$,  becomes a Banach space with the norm 
  
  \begin{align}\notag
  \|f\| + \left( \int_0^1 \left(s^{-\alpha}  \Omega_{r}(s,\>f )\right)^q   \ \cfrac{ds}{s}\right)^{1/q}. 
  \end{align}
 We use the standard convention for the definition of norm when $q=\infty$. \\
 

{\it Remark.} Notice that the Besov norm dose not depend on $r$ due to the  monotonicity of the modulus of continuity. Therefore, for fixed $q$ and $\alpha$,  the all  definitions of Besov spaces in terms of modulus of continuity are equivalent,     so they are  independent of choice of $r$. \\ 

By the definition of  the Besov spaces, it  is trivial that the following inclusions hold.  
\begin{align}\notag
\mathcal D(A)\subset B_q^\alpha\subset \mathcal H , 
\end{align}
  where $\mathcal D(A)$ is domain of  the operator $A$ and is dense in $\mathcal H$ under the assumptions.   Also, note that, the space $\mathcal D(A)$ is a Banach space with the norm $\|f\|+\|Af\|$. (See for example, \cite{Hille-Philips}.)\\
  
 By the spectral theory, the operator $A$ has a representation
 
\begin{align}\notag
A= \int_0^\infty \xi dP_{\xi} ~,
\end{align}
where $dP_{\xi}$ is a spectral measure. Therefore, for any  $f\in \mathcal D(A)$ and $g\in  \mathcal H$ one has 

 \begin{align}\notag
 \langle A f, g\rangle = \int_0^\infty \xi~  d(P_\xi f, g) ,
 \end{align}
and hence 
the definition for  domain of  $A$  becomes equivalent to 
  $$\mathcal D(A)=\left\{ f\in  \mathcal H : ~~~ \|A f\|^2:=\int_0^\infty \xi^2 d(P_\xi f, f)<\infty\right\}.$$
   For any bounded Borel measurable function   $\beta$  on the interval $(0,\infty)$, by the spectral theorem     the operator   $\beta(A)$ is bounded with   $\|\beta(A)\|_{op}= \|\beta\|_\infty$,  and has an integral representation

\begin{align}\notag
\beta(A):=\int_0^\infty \beta(\xi)dP_{\xi}.  
\end{align}

We say a vector  (function) $f\in\mathcal H$ is a {\it  Paley-Wiener function} with respect to the operator $A$ and belongs to 
  $PW_{[a, b]}(A)$ for some $0\leq a<  b<\infty$ if $d(P_\xi f, f)$ is supported in the interval $[a, b]$.
 An 
equivalent definition of Paley-Wiener spaces $PW_{[a , b]}(\Delta)$ using the so-called functional form of the spectral theorem has been  given in \cite{BS}. \\

Associated to the operator $A$,  there exists a measure $\nu$ on $\RR^+$,  a    direct integral of Hilbert spaces
 $E= (E_\lambda)_{\lambda>0}$, and a unitary operator $\mathcal{F}$  which maps $\mathcal{H}$ onto $E$. The Hilbert space 
  $E$ contains  
  all $\nu$-measurable vector valued functions $\lambda \rightarrow
e(\lambda )\in E_\lambda $ for which 
$$\|e\|_{E}=\left ( \int ^{\infty }_{0}\|e(\lambda )\|^{2}_{E_\lambda} d\nu(\lambda) \right)^{1/2}  <\infty \ .$$
It is natural to  use the notation  $E=\int_0^\infty E_\lambda \ d\nu(\lambda)$. 
    The operator $\mathcal F$
       transforms
  domain of $A^{k}$ onto $E^{k}=\{e\in E|\lambda ^{k}e\in E\}$. 
  The norm on $E^k$ is given by   
$$\|e(\lambda )\|_{E^{k}}= \left (\int^{\infty}_{0} \lambda ^{2k} \|e(\lambda )\|^{2}_{E_\lambda}
d\nu(\lambda)  \right )^{1/2} ,$$ 
 and for any $f$ in domain of $A^k$ is $\mathcal{F}(A^{k} f)(\lambda)=\lambda ^{k}
(\mathcal{F}f)(\lambda)$.  
In analogy to the classical case, 
we call the unitary operator $\mathcal{F}$ the ``Fourier transform"  or ``Plancherel transform\rq\rq{} of $\mathcal H$,  and  $\nu$  the  ``Plancherel measure\rq\rq{}.   \\

 We say $f$ is ``bandlimited\rq\rq{} if  its Fourier transform $\mathcal{F}f$  has support in $[a, b]$. It is easy to verify that $f$ belongs to  the Hilbert space  $PW_{[a,b]}(A)$  if and only if it is  bandlimited.   
 We say a function $f$ in $\mathcal H$ is ``smooth\rq\rq{} if  it lies in $\cap_{k\in\NN} \mathcal D(A^k)$. One can easily prove  that 
  $PW_{[a,b]}(A)  \subset \cap_{k\in\NN} \mathcal D(A^k)$, so every bandlimited function is smooth. \\
  

{\it Notations.} By $\| \ \|_{\infty}$, we mean uniform norm.  In the sequel, we will drop the indices  for the norms $\|\ \|_{op}$ ,  $\| \ \|_{\infty}$, and $\| \ \|_{\mathcal H}$   when there is no confusion. And, we use   $\preceq (\succeq)$    for  $\leq (\geq)$ to up some constant. 

 \section{Littlewood-Paley decomposition}\label{ProofOfLemmas}

In this section  we will show how  to decompose    an abstract Besov function  in term of \lq\lq{}smooth\rq\rq{} and \lq\lq{}band-limited\rq\rq{} functions. The smoothness and bandlimitedness of a function is understood in terms of above concept.    
To begin with the idea, we need the following set up. 
   Let $\hat\psi_0$ be a bounded  and real valued function  such that   $\mathrm{supp} (\hat\psi_0)\subset [0, 2]$  and  $\hat\psi_0(0)=1$. Let  
   $ \hat\psi$ be another real valued  function with $\mathrm{supp}(\hat\psi)\subset [1/2, 2]$. 
   For any $j\geq 1$, put $\hat\psi_j(\xi):=\hat\psi(2^{-j}\xi)$  and assume that 
   the resolution of identity holds. 
   \begin{align}\label{resolution of identity}
 \sum_{j\geq 0} \hat\psi_j(\xi)^2= 1, \quad \forall ~ \xi\geq 0 .
\end{align}

  Note that for  $j\geq 1$, $\hat\psi_j$ is compact supported and 
  $\text{supp}(\hat\psi_j) \subseteq [2^{j-1},  2^{j+1}]$.   
 By applying the spectral theorem for $A$ in (\ref{resolution of identity}),  and under the assumptions that $\hat\psi_0$ and $\hat\psi$ are real-valued,  the following version of Calder\'on decomposition, in complete analogy to the Euclidean space, holds:

 \begin{align}\label{resolution of identity operator}
 \sum_{j\geq 0}   \hat\psi_j(A)^2f= f ,~\quad \forall f\in \mathcal H.
\end{align}
Here,  the series converges in $\mathcal H$ and 
under our assumptions,  the functions 
  $\hat\psi_j(A)f$
  are smooth and  bandlimited with ``Fourier support\rq\rq{} in ${[2^{j-1},  2^{j+1}]}$. \\

Our main result in Theorem  \ref{mainLemma1}    presents  a unified  
  description of  inhomogeneous abstract  Besov spaces  for all $\alpha>1/2$ in terms of  smooth and bandlimited functions, as follows.


 \begin{theorem}\label{mainLemma1} Given any $\alpha>1/2$, $1\leq
q\leq \infty$ and 
 $f\in \mathcal H$, 
\begin{equation} 
\|f\|_{B_q^\alpha} \asymp  \|f\|+ \left(\sum_{j=0}^{\infty}\left(2^{j\alpha
}\|\hat\psi_j(A)f\|\right)^{q}\right)^{1/q}
\label{normequiv} 
\end{equation}
with the standard modifications for $q=\infty$, provided that  both sides  are finite.
\end{theorem}

By properties of  the 
  semigroup   $\{T_t\}_{t>0}$ and that $\|\hat\psi_j(A)\|=\|\hat\psi_j\|\leq 1$, one can easily prove the following estimations for     operator norms for all   $r\in \NN$ 
  \begin{align}\label{technical lemma2} 
\| \left( I-T_\tau\right)^r\hat\psi_0(A)\| &\preceq \tau^r,\quad \text{and}\\\notag
    \| \left( I-T_\tau  \right)^r\hat\psi_j(A)\| &\preceq  \tau ^r2^{(j+1)r/2} \quad j\geq 1 .
  \end{align}
  The inequality constants are uniform in (\ref{technical lemma2}), i.e., they are  independent of $j, r, \tau, \psi$, and $ \psi_0$.  
  We use  (\ref{technical lemma2}) in the following lemma to   provide an upper estimation for the modulus of continuity $\Omega_r$. 
 We need some preparation before the lemma. Let $m\in \RR$ and $k<0$ such that $k+m\geq 0$. Put $w_j:= 2^{kj}$ and $c_j:=2^{mj}$ for $j\in \ZZ$. 
   
  %

\begin{lemma}\label{second technical lemma}
Let $r\in\NN$, $1\leq q<\infty$. Let  $k, m\in \RR$ be as in above. Then for 
any  $f\in \mathcal H$
\begin{align}\notag
\Omega_r(s,f)^{q}  \preceq s^{q r} \sum_{j=0}^\infty \left(2^{j r/2} w_j^{1/q}c_j
\| \hat\psi_j(A)f\|\right)^{q} \quad \quad \forall ~ s\in (0,1) . 
\end{align}
The estimation  holds for $q=\infty$ by modification.
 \end{lemma}

 \begin{proof} 
Take 
 $0<\tau\leq s< 1$  and 
 let $f\in \mathcal H$.  By Applying the discrete version of  Calderon decomposition in  (\ref{resolution of identity operator}), 
 for $f$ we have 
 \begin{align}\label{x}
   \|\left(I-T_\tau \right)^r f\|
 \leq  \sum_{j= 0}^\infty \|\left(I-T_\tau \right)^r\hat\psi_j(A)^2f\|~.
\end{align}
So, 

\begin{align}\notag
 \Omega_r(s,f)^{q} :=\left(\sup_{0<\tau\leq s} \|\left(I-T_\tau \right)^r f\|\right)^q
 &\leq  \left( \sum_{j=0}      \sup_{0<\tau\leq s}\|\left(I-T_\tau \right)^r\hat\psi_j(A)^2f\|\right)^{ q} \\\notag
 &\leq \left( \sum_{j=0}     w_jc_j \sup_{0<\tau\leq s}\|\left(I-T_\tau \right)^r\hat\psi_j(A)^2f\|\right)^{ q}
 \end{align}
The   H\"older inequality  and  the relation    (\ref{technical lemma2})    imply that

\begin{align}\notag
 \Omega_r(s,f)^{q}  
&\preceq    \sum_{j=0} w_j    \left(c_j\sup_{0<\tau\leq s}\|\left(I-T_\tau \right)^r\hat\psi_j(A)^2f\|\right)^{ q}\\\notag 
&\preceq    \sum_{j=0} (w_j^{1/q} c_j~ \|\hat\psi_j(A)f\|)^q~ \left(\sup_{0<\tau\leq s}\|\left( I-T_\tau  \right)^r \hat\psi_j(A)\|\right)^{q}\\\notag 
&\preceq   s^{{q}r} \sum_{j=0} \left(2^{jr/2} w_j^{1/q} c_j~ \|\hat\psi_j(A)f\|\right)^q ~.
\end{align}
 
 This proves the assertion of the lemma for $q>1$. 
For $q=1$   and $q=\infty$, the proof can be easily obtained by  some modifications and similar arguments.
 \end{proof}
 
 The decay property of $\Omega_r$  near the origin is a  consequence of the preceding lemma as  follows. 
\begin{corollary} For any $r\in \NN$ 
 $$\Omega_r(s,f)=\mathcal O(s^r)\quad \quad 0<s<1$$

\end{corollary}

\section{Proof of  Theorem \ref{mainLemma1}}\label{Proof of main Theorem}

Theorem  \ref{mainLemma1} is a  result of   Theorems \ref{partI} and \ref{partII}. 

 \begin{theorem}\label{partI}
 Let $1\leq q<\infty$ and $\alpha>1/2$.  
   Given $f\in \mathcal H$, if $\{2^{j\alpha} \hat\psi_j(A)f\}_j\in l^q(\ZZ^+, \mathcal H)$, then $f\in B_{q}^\alpha$ and

 \begin{align}\label{fraction-norm}
 \int_0^1  \left(s^{-\alpha }  \Omega_r(f,s)\right)^q ds/s  
\preceq
 \sum_{j\in\ZZ^+} ( 2^{j\alpha} \| \hat\psi_j(A)f\|)^q  ~ , 
\end{align}
provided that the expression on the right is finite. 
The inequality also hold for $q=\infty$ with modification. 
\end{theorem}

\begin{proof}  Pick    $r\in \NN$ such that $r\leq 2\alpha$. Let 
$k$ and $m$ be as in above. Furthermore, assume  that $k+mq\leq q(\alpha-r/2)$.  (Note that there is  a large class   of   pairs  $(k,m)$ such that $k+m\geq 0$ and satisfy    the inequality.)   Thus by Lemma \ref{second technical lemma} and $w_j=2^{kj}$ and $c_j=2^{mj}$,  we  can write the following.
 
\begin{align}\notag
   \int_0^1 \left( s^{-\alpha}     \Omega_r(f,s)\right)^q ds/s 
  \preceq
  \sum_{j=0} \left(2^{jr/2} w_j^{1/q} c_j~  \|\hat\psi_j(A)f\|\right)^q   
 \preceq 
 \sum_{j=0} \left(2^{j\alpha}    \|\hat\psi_j(A)f\|\right)^q  ~ , 
\end{align}

and this 
 completes  the proof of the theorem for $1\leq q<\infty$. For $q=\infty$, the proof  can be easily obtained by some modifications. 

 \end{proof}
 
 The next theorem completes the proof of Theorem \ref{mainLemma1}. 
 
 \begin{theorem}\label{partII}  Let $1\leq q<\infty$ and $\alpha>1/2$.   
Given $f\in \mathcal H$, $r\in \NN$  such that $r<2\alpha$,  we have 

  \begin{align}\label{converse of first theorem}
  \sum_{j\in\ZZ^+} ( 2^{j\alpha} \|\hat\psi_j(A)f\|)^q  \preceq  \int_0^1  \left(s^{-\alpha }  \Omega_r(f,s)\right)^q ds/s .
\end{align}
The inequality also hold for $q=\infty$ with modification.
\end{theorem}

\begin{proof}
Let  $u$ be the  function that we had  in Section \ref{notations},  and $T_t=u(tA)$.  For all $\lambda>0$, define 

\begin{align}\label{G-function}
G(\lambda)=\lambda^{-n} \int_0^1    s^{-\alpha+r+1} (1- u(s\lambda))^{r}ds/s . 
\end{align}
Here, $n$ is  a  large number that can be fixed later. 
Therefore,  by  applying functional calculus theory for (\ref{G-function}), for  all $g\in \mathcal H$ 
  \begin{align}\label{operator inequality-for-g}
    G(A)g  =   \int_0^{1}  s^{-\alpha+r+1}  A^{-n}(I-T_s)^{r} g  
   ~ ds/s ~. 
   \end{align}
  If we substitute  $g$ by 
   $\hat\psi_j(A)f$   in  (\ref{operator inequality-for-g}) and then use H\"older inequality, 
 we obtain the following inequalities. 
 \begin{align}\notag
    \|G(A)\hat\psi_j(A)f\| &\leq    \int_0^{1}  s^{-\alpha} \|A^{-n}(I-T_s)^{r}\hat\psi_j(A)f\| 
   ~ s^{r+1} ds/s\\\notag
   & \leq \int_0^{1}  s^{-\alpha} \|(I-T_s)^{r}f\| \ \|A^{-n}\hat\psi_j(A)\|  ~ s^{r+1} ds/s\\\notag
   & \leq \left(\int_0^{1} \left(s^{-\alpha} \|(I-T_s)^{r}f\|\right)^q ~ds/s\right)^{1/q}
\left(\int_0^{1}  s^{(r+1)q\rq{}} (2^{-nq'(j-1)} )~ds/s\right)^{1/q\rq{}}\\\notag
   & \preceq 2^{-n(j-1)}\left(\int_0^{1} \left(s^{-\alpha} \Omega_r(s,f)\right)^q ~ds/s\right)^{1/q}  ~.
   \end{align}
  Therefore we proved that  
   \begin{align}\label{eq:1}
 \|G(A)\hat\psi_j(A)f\|^q 
\leq 2^{-n(j-1)q} \int_0^{1} \left(s^{-\alpha} \Omega_r(s,f)\right)^q ~ds/s ~.
 \end{align}
 From the other side, with 
   a simple calculation  one can show that 
 \begin{align}\label{eq:2}
 \|G(A)\hat\psi_j(A)f\|\succeq \  2^{j(-n+1-\alpha+r)} \|\hat\psi_j(A)f\|~ . 
 \end{align}
 As  result of    (\ref{eq:1}) and (\ref{eq:2}), we get
 
 \begin{align}
 2^{j\alpha q} \|\hat\psi_j(A)f\|^q &\leq 2^{j\alpha q}  2^{-jq(-n+1-\alpha+r)} \|G(A)\hat\psi_j(A)f\|^q \\\notag
& \preceq 2^{j(-2n+1+r)q} \int_0^{1} \left(s^{-\alpha} \Omega_r(s,f)\right)^q ~ds/s. 
 \end{align}
 
And, hence 
 \begin{align}
\sum_{j\geq 0} 2^{j\alpha q}  \|\hat\psi_j(A)f\|^q 
 \leq \left(\sum_{j\geq 0}  2^{j(-2n+1+r)q}\right) \int_0^{1} \left(s^{-\alpha} \Omega_r(s,f)\right)^q ~ds/s  .
 \end{align}
 This proves the theorem  if we  let  $n>\cfrac{r+1}{2}$ . 
\end{proof}

{\bf Questions.} These questions are still open  and need to be answered.  Can we remove the restriction $\alpha>1/2$ and prove   Theorem  \ref{mainLemma1}
  for  $0<\alpha <1/2$? Can we prove Theorem \ref{mainLemma1}   for $0<q<1$?

 \makeatletter
\renewcommand{\@biblabel}[1]{\hfill#1.}\makeatother

   \end{document}